\documentclass[12pt]{article}
\usepackage{latexsym, amsthm, amsmath, amssymb}
\newcommand{\R}{\mathbb{R}}
\newcommand{\N}{\mathbb{N}}
\newcommand{\No}{\mathbb{N}_0}
\newcommand{\Z}{\mathbb{Z}}
\newcommand{\Q}{\mathbb{Q}}
\newcommand{\C}{\mathbb{C}}

\newtheorem{lemma}{Lemma}[section]

\newtheorem{theorem}[lemma]{Theorem}
\newtheorem{corollary}[lemma]{Corollary}\newtheorem{proposition}[lemma]{Proposition}

\title{Generalizations of Khovanski\u\i's theorems on growth of sumsets 
in abelian semigroups}
\author{V\'\i t Jel\'\i nek and Martin Klazar\thanks{Institute 
for Theoretical Computer Science and Department of Applied Mathematics, 
Faculty of Mathematics and Physics of Charles University, 
Malostransk\'e n\'am\v est\'\i\ 25, 118 00 Praha, Czech Republic. 
ITI is supported by the project 1M0021620808 of the Czech Ministry 
of Education.
Email: {\tt \{jelinek, klazar\}@kam.mff.cuni.cz}}}
\date{\today}
\begin{document}\maketitle
\begin{abstract}
We show that if $P$ is a lattice polytope in the nonnegative orthant of $\R^k$ and $\chi$ is a
coloring of the lattice points in the orthant such that the color $\chi(a+b)$ depends only on 
the colors $\chi(a)$ and $\chi(b)$, then the number of colors of the lattice points
in the dilation $nP$ of $P$ is for large $n$ given by a polynomial (or, for rational $P$, by a quasipolynomial). 
This unifies a classical result of Ehrhart and Macdonald on lattice points in polytopes 
and a result of Khovanski\u\i{} on sumsets in semigroups. We also prove a strengthening of multivariate 
generalizations of Khovanski\u\i's theorem. Another result of Khovanski\u\i{} states that the 
size of the image of a finite set after $n$ applications of 
mappings from a finite family of mutually commuting mappings is for large $n$ a polynomial. 
We give a combinatorial proof of a multivariate generalization of this theorem.
\end{abstract}
\section{Introduction}

In many classes of enumerative combinatorial problems, every counting 
function is equal---usually for sufficiently large arguments---to a polynomial or to a quasipolynomial. In this article, we consider several
classes of problems with this property, (re)derive their polynomiality 
in a more uniform manner, and generalize and strengthen existing results. We begin 
with three important examples.

\subsection{Lattice polytopes, sumsets in semigroups, ideals in a poset}

For $n\in\N$ and a lattice polytope $P\subset\R^k$, which is a convex hull of a finite set 
of points from $\Z^k$, denote by $i(P,n)$ 
the number of the lattice points lying in the dilation $nP=\{nx:\;x\in P\}$ of $P$,
$$
i(P,n)=|nP\cap \Z^k|.
$$
Ehrhart and Macdonald obtained the following result.

\begin{theorem}[Ehrhart \cite{ehrh}, Macdonald \cite{macd63, macd71}]\label{mrbody}
The number $i(P,n)$ of the lattice points in $nP$ is for \emph{all $n\in\N$} given by a 
polynomial. 
\end{theorem}

\noindent
More generally, if $P$ is a rational polytope (its vertices have rational coordinates), 
then $i(P,n)$ is for all $n\in\N$ given by a quasipolynomial (the definition of a 
quasipolynomial is recalled in Section 1.3). See Stanley 
\cite[Section 4.6]{stanEC1} for more information.

For a commutative semigroup $(G,+)$ and subsets $A,B\subset G$, consider the sumsets 
$$
n*A=\{a_1+\cdots+a_n:\;a_i\in A\}\ \mbox { and }\ 
A+B=\{a+b:\;a\in A, b\in B\}.
$$
For a (typically infinite) set $X$, its subset $B\subset X$, and a family 
${\cal F}$ of mutually commuting mappings $f\colon X\to X$, the $n$th iterated image of $B$ by ${\cal F}$ is
\[
{\cal F}^{(n)}(B)=\bigcup_{f_i\in{\cal F}}(f_1\circ\dots\circ f_n)(B),
\]                                                                    
where $f(B)$ denotes the set $\{f(x):\,x\in B\}$.
The following three theorems are due to Khovanski\u\i.

\begin{theorem}[Khovanski\u\i\ \cite{khov1}]\label{summnoz}
Let $A$ and $B$ be finite sets in a commutative semigroup.
\begin{enumerate}
\item For large $n$, the cardinality of the sumset $|n*A|$ is  given by a polynomial. 
\item For large $n$, the cardinality of the sumset $|n*A+B|$ is  given 
by a polynomial. 
\end{enumerate}
\end{theorem}

\begin{theorem}[Khovanski\u\i\ \cite{khov2}]\label{character}
Let $G=(G,+)$ be a commutative semigroup, $A,B\subset G$ be two finite subsets, and 
$\psi:\;G\to\C$ be an additive character of $G$ (i.e., $\psi(a+b)=\psi(a)\psi(b)$). 
Then there exist polynomials $p_a(x)$, $a\in A$, such that for large $n$ one has
$$
\sum_{a\in n*A+B}\psi(a)=\sum_{a\in A}p_a(n)\psi(a)^n.
$$
\end{theorem}

\begin{theorem}[Khovanski\u\i\ \cite{khov1}]\label{iterimage}
If $B$ is a finite subset of $X$ and ${\cal F}$ is finite family of mutually commuting 
mappings from $X$ to itself, then the cardinality of the 
iterated image ${\cal F}^{(n)}(B)$ is for large $n$ given by a polynomial in $n$.
\end{theorem}

\noindent Khovanski\u\i{} stated and proved just part 2 of 
Theorem~\ref{summnoz} (as a corollary of Theorem~\ref{iterimage}); however, 
part 2 immediately implies part 1 which we state explicitly for the purpose 
of later reference. Both Theorem~\ref{character} and Theorem~\ref{iterimage} 
include part 2 of Theorem~\ref{summnoz} as a particular case: set $\psi\equiv 
1$, respectively set $X=G$ and consider the mappings ${\cal F}=\{s_a:\;a\in 
A\}$ where $s_a(x)=x+a$. 

Let us now consider the poset $(\No^k,\le)$, $\No=\{0,1,2,\dots\}$, with 
componentwise ordering: 
$$ a=(a_1,\dots,a_k)\le b=(b_1,\dots,b_k)\iff a_i\le 
b_i,\ i=1,\dots,k. $$ 
A lower ideal $S\subset\No^k$, is a set satisfying the condition 
$a\le b, b\in S\Rightarrow a\in S$. The following result was first posed as a 
problem in the American Mathematical Monthly, see also \cite[Exercise 6 in 
Chapter 4]{stanEC1}.

\begin{theorem}[Stanley \cite{stan7576}]\label{dolniid}
For a lower ideal $S$ in the poset $(\No^k,\le)$, the number of the elements 
$a=(a_1,\dots,a_k)\in S$ with $\|a\|_1=a_1+\cdots+a_k=n$ is for large $n$ given by a polynomial.
\end{theorem}

We prove all five theorems (Theorem~\ref{mrbody} in a weaker form for large
$n$ only) in the framework of more general results in Section 2.

\subsection{Our results}
At first, we wanted to understand the connection between Theorems \ref{mrbody} and 
\ref{summnoz}, and to find reasons for polynomiality of these two and other 
classes. This turned into a goal to explain the above results on polynomiality in 
a uniform manner, and to give combinatorial proofs of these combinatorial results; 
some of the above theorems were originally proved by somewhat opaque algebraic arguments. 
We succeeded in this 
to large extent for the five theorems. In Section 2, we demonstrate that 
Theorems \ref{mrbody}--\ref{iterimage} (Theorem~\ref{mrbody} for large
$n$ only) follow as corollaries of Stanley's Theorem \ref{dolniid} or of its natural extensions stated in Theorems~\ref{stanleyho} and 
\ref{zobstanleyho}. We will give 
multivariate generalizations of Theorems~\ref{summnoz}--\ref{iterimage}.  
Theorem~\ref{zobstanleyho} can be used to prove polynomiality of further 
classes of enumerative problems, which we briefly mention in Section 3 and will discuss in 
details in \cite{jeli_klaz_next}.  

We build on the results of Khovanski\u\i{} \cite{khov1, khov2}, Nathanson and 
Ruzsa \cite{nath_ruzs} and Stanley \cite{stan7576}. Khovanski\u\i's original 
proof of part 2 of Theorem \ref{summnoz} as a corollary of Theorem 
\ref{iterimage} in \cite{khov1} was algebraic, by means of the Hilbert 
polynomial of graded modules. In \cite{khov2}, he gave a combinatorial proof 
of part 2 as a corollary of Theorem \ref{character}. Extending 
Khovanski\u\i's algebraic argument, Nathanson \cite{nath} proved a 
multivariate generalization of part 2 (see Theorem \ref{affchov}). Then 
Nathanson and Ruzsa \cite{nath_ruzs} gave a simple combinatorial proof for a 
multivariate generalization of part 1 (see Theorem \ref{multivar}). 

Our contribution is a common strengthening of these generalizations in 
Theorem~\ref{stron_ev_poly_chov}: If $A_1,\dots, A_l$ are finite sets in a 
commutative semigroup $(G,+)$ and $$ p(n_1,\dots, 
n_l):=|n_1*A_1+\cdots+n_l*A_l|, $$ then there is a constant $c>0$ such that 
for any $l$-tuple of arguments $n_1,\dots, n_l$, if the arguments $n_i$ not 
exceeding $c$ are fixed, then $p(n_1,\dots,n_l)$ is a polynomial function in 
the remaining arguments $n_i$ bigger than $c$. We characterize such 
eventually strongly polynomial functions in 
Proposition~\ref{event_poly_right}. 

In Theorems \ref{nasepoly} and  \ref{naseqpoly}, we prove our next result, a common 
generalization of a weaker form of Theorem \ref{mrbody} and part 1 of Theorem \ref{summnoz}. 
We prove that if $P$ is a lattice polytope in the nonnegative orthant of $\R^k$, and $\chi$ is 
a coloring of the lattice points in the orthant such that $\chi(a+b)$ depends only 
on the colors $\chi(a)$ and $\chi(b)$, then the number of colors 
$$
|\chi(nP\cap\Z^k)|
$$
used on the points $nP\cap\Z^k$
is a polynomial in $n$ for large $n$. More generally, if $P$ is a rational polytope, then 
the number of colors is for large $n$ a quasipolynomial (Theorem~\ref{naseqpoly}). 
This includes  Theorem \ref{mrbody} (in a weaker form for large $n$) and part 1 of 
Theorem \ref{summnoz} as particular cases. We want to remark that our Theorem \ref{nasepoly} 
is to some extent hinted to already by Khovanski\u\i{} \cite[paragraph 5]{khov1} who derives, 
as an application of part 2 of Theorem~\ref{summnoz}, the weaker form of Theorem \ref{mrbody}. 
We also obtain Theorem~\ref{nasepoly} as a corollary of part 2 of Theorem \ref{summnoz} 
and a geometric lemma.

Our third result are multivariate generalizations of Theorems \ref{character} and 
\ref{iterimage}, presented in Theorems \ref{charactermulti} and \ref{iterimagemulti}, respectively. 
We give combinatorial proofs. The proof of Theorem \ref{charactermulti} on additive characters 
is a simple extension of the combinatorial proof of Theorem~\ref{stron_ev_poly_chov} and we 
only give a sketch of the proof. 
The proof of Theorem \ref{iterimagemulti} on iterated images is more interesting. We derive 
it from Theorem \ref{zobstanleyho} which extends Stanley's Theorem \ref{dolniid} on lower 
ideals. Theorem \ref{zobstanleyho} characterizes the sets $S\subset\No^k$ for which 
Theorem \ref{dolniid} holds. 

Our combinatorial approach is based on expressing counting problems in terms of 
colorings $\chi$ of $\No^k$ and on counting the color classes of 
$\chi$ via appropriate representatives, so called substantial points. We have learned 
both techniques from Nathanson and Ruzsa \cite{nath_ruzs}. A new ingredient is
the representation of counting functions in a compact and convenient way by their generating 
power series (which play almost no role in \cite{khov1, khov2, nath, nath_ruzs}). We recall 
some results on them in the next subsection.

In Section 3, we give some concluding remarks and references to further examples of 
polynomial classes of enumerative problems. 

\subsection{Notation and results on power series}

We fix notation and recall some useful results on power series. $\N$ is the 
set of natural numbers $\{1,2,\dots\}$ and $\No$ is the set $\{0\}\cup\N$. 
The symbols for number sets $\Z$, $\Q$, $\R$, and $\C$ have their usual 
meanings. For $n\in\N$, the set $\{1,2,\dots,n\}$ is denoted by $[n]$. We call 
the elements of $\Z^k$ {\em lattice points}. All semigroups in this article 
are commutative. We will use the lexicographic ordering of $\No^k$, which is 
a total ordering: $a<_{lex}b$ iff $a_1=b_1,\dots,a_i=b_i,a_{i+1}<b_{i+1}$ for 
some $i$, $0\le i<k$.

A {\em quasipolynomial} is a function $f\colon\Z\to\C$ for which there are 
$d$ polynomials $p_1(x),\dots, p_d(x)$ such that $f(n)=p_i(n)$ if $n\equiv i\ 
{\rm mod }\ d$; $d$ is the {\em period} of $f$. Equivalently, 
$f(n)=a_k(n)n^k+\dots+a_1(n)n+a_0(n)$ where $a_i\colon\Z\to\C$ are periodic 
functions. The term quasipolynomial is sometimes (e.g., in \cite{khov2}) used 
also for linear combinations of exponentials with polynomial coefficients (as 
in Theorem \ref{character}); we use it in the present sense. 

We shall use formal power series 
$$
F(x_1,\dots,x_k)=\sum_{a\in\No^k}\alpha(a)x_1^{a_1}\dots x_k^{a_k}
$$
with real coefficients $\alpha(a)=\alpha(a_1,\dots,a_k)$ and several variables 
$x_1,\dots,x_k$; their set is denoted by $\R[[x_1,\dots,x_k]]$. The symbol
$$
[x_1^{a_1}\dots x_k^{a_k}]\;F
$$ 
denotes the coefficient $\alpha(a_1,\dots,a_k)$ of $x_1^{a_1}\dots x_k^{a_k}$ in 
$F$. For a subset $A\subset\No^k$, $F_A(x)=F_A(x_1,\dots,x_k)\in\R[[x_1,\dots,x_k]]$ 
is the power series
$$
F_A(x_1,\dots,x_k)=\sum_{a\in A}x_1^{a_1}\dots x_k^{a_k},
$$
i.e., $\alpha(a)$ is the characteristic function of $A$. 

\begin{lemma}\label{event_poly}
Let $F\in\R[[x_1,\dots,x_k]]$ be a rational power series of the form
$$
F(x_1,\dots,x_k)=\frac{r(x_1,\dots,x_k)}{(1-x_1)^{e_1}\dots(1-x_k)^{e_k}}
$$
where $r\in\R[x_1,\dots,x_k]$ is a polynomial and $e_i\in\No$. 
Then for every $l\in\No$, $l\le k$, and every $l$-tuple
$(a_1,\dots,a_l)\in\No^l$, there exist a constant $c>0$ and a polynomial 
$p\in\R[x_{l+1},\dots,x_k]$ (for $l=k$ we understand $p$ as a real constant) 
such that if $n_{l+1},\dots,n_k\in\N$ are all bigger than $c$, then
$$
[x_1^{a_1}\dots x_l^{a_l}x_{l+1}^{n_{l+1}}\dots x_k^{n_k}]\;F=p(n_{l+1},\dots,n_k).
$$
\end{lemma}
\begin{proof}
Let us check that the claim holds when $k=1$, $0\le l\le 1$, and $r(x_1)=r(x)=x^b$. By the binomial expansion,
$$
\frac{x^b}{(1-x)^e}=\sum_{n\ge 0}\binom{n+e-1}{e-1}x^{b+n}=
\sum_{n\ge b}\binom{n+e-1-b}{e-1}x^{n}. 
$$
The general case reduces to this by 
expressing $F$ as a  finite linear combination of terms of the type
$$
\frac{x_1^{b_1}\dots x_k^{b_k}}{(1-x_1)^{e_1}\dots(1-x_k)^{e_k}}=
\prod_{i=1}^k\frac{x_i^{b_i}}{(1-x_i)^{e_i}}.
$$
\end{proof}

\noindent
We add three comments to the lemma. If the polynomial $r(x_1,\dots,x_k)$ has rational 
coefficients, then $p(x_{l+1},\dots,x_k)$ has rational coefficients as well. Also, 
Lemma~\ref{event_poly} holds more generally for 
any subset of the set of variables $x_1,\dots,x_k$ (we have chosen the subset 
$x_{l+1},\dots,x_k$ only for the convenience of notation). 
Finally, Lemma~\ref{event_poly} can be strengthened by selecting the constant
$c$ first and thus making it independent on the $l$-tuples $(a_1,\dots,a_l)$. We 
return to this matter in Proposition~\ref{event_poly_right}.

Let $F\in\R[[x_1,\dots,x_k]]$ be a power series and $P=\{P_1,\dots,P_l\}$ be a partition 
of the index set $[k]$ into $l$ blocks. The substitution $x_i:=y_j$, 
where $1\le i\le k$ and $j$ is the unique index satisfying $i\in P_j$, turns
$F$ into the power series $G\in\R[[y_1,\dots,y_l]]$ with the 
coefficients
$$
[y_1^{n_1}\dots y_l^{n_l}]\;G=\sum[x_1^{a_1}\dots x_k^{a_k}]\;F,
$$
where we sum over all $a\in\No^k$ satisfying $\sum_{i\in P_j}a_i=n_j$, $1\le j\le l$.
We call a substitution of this kind {\em $P$-substitution}. It is immediate that
$P$-substitutions preserve the class of rational power series considered in 
Lemma~\ref{event_poly}.

\begin{lemma}\label{bsubst}
If $F\in\R[[x_1,\dots,x_k]]$ has the form $F=r(1-x_1)^{-e_1}\dots(1-x_k)^{-e_k}$,
where $r\in\R[x_1,\dots,x_k]$ and $e_i\in\No$, and 
$G\in\R[[y_1,\dots,y_l]]$ is obtained from $F$ by a $P$-substitution, then
$G=s(1-y_1)^{-f_1}\dots(1-y_l)^{-f_l}$, where $s\in\R[y_1,\dots,y_l]$ and 
$f_i\in\No$. 
\end{lemma}

\section{Generalizations of Khovanski\u\i's theorems}

This section is devoted to the proofs of our main results, which are Theorems 
\ref{nasepoly}, \ref{naseqpoly}, \ref{stron_ev_poly_chov}, \ref{charactermulti}, 
\ref{iterimagemulti}, and \ref{zobstanleyho}. 

\subsection{Additive colorings}

We shall work with the semigroup $(\No^k,+)$, where the addition of $k$-tuples is defined componentwise. For a (possibly infinite) set of colors $X$, we say that a coloring 
$\chi:\;\No^k\to X$ is {\em additive} if
$$
\chi(a+b)=\chi(c+d)\mbox{ whenever $\chi(a)=\chi(c)$ and $\chi(b)=\chi(d)$},
$$ 
that is, if the color of every sum depends only on the colors of summands. The coloring
$\chi$ then can be viewed as a homomorphism between the semigroups (in fact monoids) 
$(\No^k,+)$ and $(X,+)$. The additivity of $\chi$ 
is equivalent to the seemingly weaker property of {\em shift-stability}, which only requires 
that 
$$
\chi(a+b)=\chi(c+b) \mbox{ for every $b$ whenever $\chi(a)=\chi(c)$}. 
$$
Indeed, if 
$\chi$ is shift-stable and $a,b,c,d\in \No^k$ are arbitrary elements satisfying 
$\chi(a)=\chi(c)$ and $\chi(b)=\chi(d)$, then $\chi(a+b)=\chi(a+d)$ and $\chi(a+d)=\chi(c+d)$, 
so $\chi(a+b)=\chi(c+d)$.

Let $(G,+)$ be a (commutative) semigroup, we may assume that it has a neutral element and is 
a monoid. If $A=(a_1,\dots,a_k)$ is a sequence of (possibly repeating) elements from $G$, then 
the \emph{associated coloring}  
$$
\chi\colon\No^k\to G,\ \chi(v)=\chi((v_1,\dots,v_k))=v_1a_1+\cdots+v_ka_k,
$$ 
is additive. In terms of this coloring, the cardinality of the sumset 
$$
n*A=\{n_1a_1+\dots+n_ka_k\colon n_1+\cdots+n_k=n\}
$$ 
equals to the number of colors $|\chi(nP\cap\Z^d)|$ appearing on the lattice points
in the dilation of the unit simplex 
$$
P=\{x\in\R^k\colon x_i\ge 0, x_1+\cdots+x_k=1\}.
$$

We prove the following common generalization of a weaker form of Theorem \ref{mrbody} (for 
large $n$ only) and part 1 of Theorem \ref{summnoz}.

\begin{theorem}\label{nasepoly}
Let $P$ be a polytope in $\R^k$ with vertices in $\No^k$ and 
let $\chi:\,\No^k\to X$ be an additive coloring. Then, for $n\in\N$ 
sufficiently large, the number of colors
$$
|\chi(nP\cap\Z^k)|=|\chi(nP\cap\No^k)|
$$ 
is given by a polynomial.
\end{theorem}

\noindent
For large $n$, Theorem \ref{mrbody} corresponds to the case when $\chi$ is injective 
(hence additive) and $P$ is a general polytope, 
while part 1 of Theorem \ref{summnoz} corresponds to the case when $\chi$ is a general additive 
coloring and $P$ is the unit simplex. 

We begin with proving a formally stronger version of Theorem \ref{dolniid}; our proof is 
a straightforward adaptation of that 
in \cite{stan7576}. Recall that $S\subset\No^k$ is a lower
ideal in the poset $(\No^k,\le)$ if for every $a\in\No^k$ we have $a\in S$ whenever $a\le b$  
for some $b\in S$. Upper ideals are defined similarly. 
The proof rests on the well-known result, sometimes called Dickson's lemma, 
which states
that all antichains (sets with elements mutually incomparable by $\le$) in $(\No^k,\le)$ are  
finite. This lemma is a corollary of the more general fact that if $(P,\le_P)$ and 
$(Q,\le_Q)$ are two posets which have no infinite antichains and no infinite strictly 
descending chains, then this property carries over to the product poset 
$(P\times Q,\le_{P\times Q})$ (see, e.g., Kruskal~\cite{krus}). 

\begin{theorem}\label{stanleyho}
Let $S\subset\No^k$ be a lower or an upper ideal in the poset $(\No^k,\le)$. 
Then
$$
F_S(x_1,\dots,x_k)=\frac{r(x_1,\dots,x_k)}{(1-x_1)\dots (1-x_k)}
$$
where $r(x_1,\dots,x_k)$ is an integral polynomial.
\end{theorem}
\begin{proof}
Since every upper ideal $S$ has as its complement $T=\No^k\backslash S$ a lower 
ideal and vice versa, and 
$$
F_S(x)+F_T(x)=F_{\No^k}(x)=
\frac{1}{(1-x_1)\dots (1-x_k)},
$$
it suffices to prove the result only for ideals of one kind. 
Let $S$ be an upper ideal. If $M\subset S$ is the set of the minimal 
elements in $S$, then 
$$
S=\bigcup_{a\in M}O_a,
$$ 
where $O_a=\{b\in\No^k:\;b\ge a\}$. Being an antichain, $M$ is finite by Dickson's lemma and 
$S$ is a finite union of the orthants $O_a$, $a\in M$. For any finite set $T$ of points in 
$\No^k$ we have
$$
\bigcap_{t\in T}O_t=O_s,
$$
where $s=(s_1,s_2,\dots,s_k)$ is the componentwise maximum of the 
points $t\in T$. Thus, by the principle of inclusion and exclusion, the characteristic function
of $S$ is a linear combination, with coefficients $\pm 1$, of characteristic functions of finitely many orthants $O_s$. Since each of them has generating function
$$
F_{O_s}(x)=\frac{x_1^{s_1}\dotsb 
x_k^{s_k}}{(1-x_1)\dotsb (1-x_k)}, 
$$   
we have $F_S(x)=r/((1-x_1)\dots (1-x_k))$ for some integral polynomial $r$.  
\end{proof}

\noindent
Theorem \ref{dolniid} now follows as a corollary, with the help of 
Lemmas \ref{event_poly} and \ref{bsubst} and the $P$-substitution $P=\{\{1,\dots,k\}\}$. 

Next, we prove the multivariate generalizations of Theorem \ref{summnoz} from 
\cite{nath} and \cite{nath_ruzs}; this is necessary, since we need part 2 of 
Theorem~\ref{summnoz} for the proof of Theorem~\ref{nasepoly}. In 
Corollary~\ref{NRgf} we lift the result of Nathanson and Ruzsa to the level 
of generating functions. 

Suppose that $P$ is a partition of $[k]$ into $l$ blocks and $\chi\colon\No^k\to X$ is 
a coloring. For $x\in\No^k$ we define $\|x\|_P$ to be the $l$-tuple $(c_1,\dots,c_l)\in\No^l$,
where $c_i=\sum_{j\in P_i}x_j$ is the sum of the coordinates with indices in the $i$th 
block. Using the notion introduced in \cite{nath_ruzs}, we say that a point 
$a\in\No^k$ is $P$-\emph{substantial} (with respect to $\chi$) if it is the lexicographically minimum element in the set 
$$
\{b\in\No^k:\;\chi(b)=\chi(a), \|b\|_P=\|a\|_P\}. 
$$
Note that every nonempty intersection of a color class with the set $\{x\in\No^k:\;
\|x\|_P=(n_1,\dots,n_l)\}$ (for $l=1$ this is the dilation $n_1P$ where $P$ is the unit simplex) 
contains exactly one $P$-substantial point. $P$-substantial points are representatives which 
enable us to count the color classes.

\begin{corollary}\label{NRgf}
Let $P$ be a partition of $[k]$ into $l$ blocks, $\chi\colon\No^k\to X$ be an additive coloring 
and $S\subset\No^k$ be the set of $P$-substantial points. Then  
$$
F_S(x_1,\dots,x_k)=\frac{r(x_1,\dots,x_k)}{(1-x_1)\dotsb (1-x_k)}
$$
where $r(x_1,\dots,x_k)$ is an integral polynomial.
\end{corollary}
\begin{proof}
In view of the previous theorem, it suffices to show that $P$-substantial points form 
a lower ideal or, equivalently, that their complement is an upper ideal. 
The latter way is a more natural choice. Let $b\in\No^k$ be any point 
such that $b\ge a$ for a non-$P$-substantial point $a\in\No^k$. 
There is a point $a'\in\No^k$ satisfying $\chi(a')=\chi(a)$, 
$\|a'\|_P=\|a\|_P$, and $a'<_{lex}a$. Consider the point 
$b'=a'+(b-a)$. We have $\chi(b')=\chi(b)$ by the additivity (indeed, shift-stability) of 
$\chi$, and  $\|b'\|_P=\|a'\|_P+\|b-a\|_P=\|a\|_P+\|b-a\|_P=\|b\|_P$ and $b'<_{lex}b$ by the properties of addition in $(\No^k,+)$. Thus $b$ is not $P$-substantial either.    
\end{proof}
\begin{theorem}[Nathanson and Ruzsa \cite{nath_ruzs}]\label{multivar}
Let $A_1,\dots,A_l$ be finite sets in a semigroup $(G,+)$. There exist 
a constant $c>0$ and an integral polynomial $p\in\Z[x_1,\dots,x_l]$ such that if 
$n_1,\dots,n_l\in\N$ are all bigger than $c$, then
$$
|n_1*A_1+\cdots+n_l*A_l|=p(n_1,\dots,n_l).
$$
\end{theorem}
\begin{proof}
Let $A=(a_1,\dots,a_k)$ be a fixed ordering of all elements appearing in the sets 
$A_1,\dots,A_l$ (taken with their multiplicities, so $k=|A_1|+\dots+|A_l|$) and $P$ be 
the corresponding partition of $[k]$ into $l$ blocks. Let $\chi\colon\No^k\to G$ be the coloring associated with $A$ and $S\subset\No^k$ be the corresponding set of $P$-substantial points. 
Let $G\in\R[[y_1,\dots,y_l]]$ be the power series obtained from $F_S(x_1,\dots,x_k)$ 
by the $P$-substitution. Then
$$
|n_1*A_1+\cdots+n_l*A_l|=[y_1^{n_1}\dots y_l^{n_l}]\;G.
$$
The result now follows by Corollary~\ref{NRgf} and by Lemmas~\ref{event_poly} and \ref{bsubst}. 
\end{proof}

\noindent
Extending Khovanski\u\i's original algebraic argument, Nathanson \cite{nath} proved 
a multivariate generalization of part 2 of Theorem~\ref{summnoz}. 

\begin{theorem}[Nathanson \cite{nath}]\label{affchov}
Let $A_1,\dots,A_{l+1}$ be finite sets in a semigroup $(G,+)$. There exist a 
constant $c>0$ and a polynomial $p\in\Z[x_1,\dots,x_l]$ such that if
$n_1,\dots,n_l\in\N$ are all bigger than $c$, then
$$
|n_1*A_1+\cdots+n_l*A_l+A_{l+1}|=p(n_1,\dots,n_l).
$$
\end{theorem}
\begin{proof}
The proof is almost identical to the proof of Theorem~\ref{multivar}. We again see that
$$
|n_1*A_1+\cdots+n_l*A_l+A_{l+1}|=[y_1^{n_1}\dots y_l^{n_l}y_{l+1}]\;G
$$
and use Corollary~\ref{NRgf} and Lemmas~\ref{event_poly} and \ref{bsubst}.
\end{proof}

\noindent
The last ingredient needed for the proof of Theorem~\ref{nasepoly} is a geometric lemma. Before we state the lemma, let us point out some observations about multiples of polytopes. Let $P\subset\R^k$ be a polytope, $n\in\No$, and  $\alpha_1,\dots,\alpha_n\in\R$ be nonnegative coefficients. Clearly, $nP\subset n*P$. On the other hand,
representing points in $P$ as convex combinations of the vertices of $P$, we deduce 
the following set inclusion
\begin{equation}\label{eq-incl1}
\alpha_1 P+\dots+\alpha_n P\subset(\alpha_1+\dots+\alpha_n)P.
\end{equation}
In particular, $n*P\subset nP$ and thus $n*P=nP$. As a corollary, we obtain another set inclusion
\begin{equation}\label{eq-incl2}
(\alpha_1 P\cap\Z^k)+\dots+(\alpha_n P\cap\Z^k)\subset (\alpha_1+\dots+\alpha_n)P\cap\Z^k.
\end{equation}
In particular, $n*(P\cap\Z^k)\subset nP\cap\Z^k$. 
The opposite inclusion in general does not hold. To get equality in some form also 
for lattice points, we use Carath\'eodory's theorem. This theorem says that if a point $a$ in $\R^k$ is in the convex hull of a set of points 
$M$, then $a$ can be expressed as a convex combination of at most $k+1$ points of the set $M$ 
(see, e.g., Matou\v sek \cite{mato}). 

\begin{lemma}\label{geomlemma}
Let $k\in\N$ and $P\subset\R^k$ be a lattice polytope. 
Then for every $n\in\N$, $n\ge k$, we have in $(\Z^k,+)$ the identity 
$$
nP\cap\Z^k=(n-k)*(P\cap\Z^k)+(kP\cap\Z^k).
$$
\end{lemma}
\begin{proof}
Let $v_1,\dots, v_r$ be the vertices of $P$ and let $p\in nP\cap\Z^k$
with $n\in\N$ and $n\ge k$. Clearly, $p$ is in the convex hull of the points
$nv_1,\dots, nv_r$. By Carath\'eodory's theorem, $p$ is a convex combination of 
at most $k+1$ of these points. Hence 
\begin{eqnarray*}
p&=&\beta_1nw_1+\dots+\beta_jnw_j,\ \mbox{ where }\ \beta_i\ge 0\ \mbox{ and }\ 
 \beta_1+\dots+\beta_j=1,\\
&=&n_1w_1+\cdots+n_jw_j+w,
\end{eqnarray*}
where $n_i=\lfloor\beta_i n\rfloor\in\No$, $j\le k+1$, $w_1,\dots,w_j$ 
are some distinct vertices of $v_1,\dots,v_r$, and 
$$
w=\alpha_1 w_1+\dots+\alpha_j w_j,\ \mbox{ where }\  
\alpha_i=\beta_in-\lfloor\beta_i n\rfloor\in[0,1). 
$$
Since $w=p-(n_1w_1+\cdots+n_jw_j)$, we see that $w$ is a lattice point. 
By~\eqref{eq-incl1}, $w\in(\alpha_1+\dots+\alpha_j)P=cP$. We have 
$0\le c=\alpha_1+\dots+\alpha_j<j\le k+1$ and 
$c=\alpha_1+\dots+\alpha_j=n-(n_1+\cdots+n_j)\in\No$. Thus $c\le k$.
We conclude that $w\in cP\cap\Z^k$ where $c\in\No$, $c=n-(n_1+\cdots+n_j)$, and $c\le k$.

We split $n_1w_1+\cdots+n_jw_j$ in the individual $n_1+\cdots+n_j=n-c$ summands, 
each of them equal to some $w_i$, and merge $k-c$ of them with 
$w$ so that we obtain a point $z\in kP\cap\Z^k$ (using the inclusion~\eqref{eq-incl2} above). 
Thus we get the expression
$$
p=z_1+\cdots+z_{n-k}+z
$$ 
where $z_i\in P\cap\Z^k$ (in fact, $z_i\in\{v_1,\dots,v_r\}$) and $z\in kP\cap\Z^k$. This 
shows that 
$$
nP\cap\Z^k\subset(n-k)*(P\cap\Z^k)+(kP\cap\Z^k). 
$$
The opposite inclusion follows from~\eqref{eq-incl2}.
\end{proof}

We are ready to prove Theorem~\ref{nasepoly}.
     
\begin{proof}[Proof of Theorem~\ref{nasepoly}]
We consider the semigroup 
of color classes $(\chi(\No^k),+)$ and its subsets $A=\chi(P\cap\No^k)$ and 
$B=\chi(kP\cap\No^k)$. By Lemma~\ref{geomlemma},
$$
|\chi(nP\cap\No^k)|=|(n-k)*A+B|.
$$
By part 2 of Theorem \ref{summnoz} (or by Theorem~\ref{affchov} or by 
Theorem~\ref{stron_ev_poly_chov} in the next subsection), this quantity is 
for big $n$ a polynomial in $n-k$ and hence a polynomial in $n$.
\end{proof}

\bigskip
We generalize Theorem~\ref{nasepoly} to rational polytopes. Our argument is 
based on the following generalization of Lemma~\ref{geomlemma}.

\begin{lemma}\label{geomlemmamodm}
Let $k\in\N$ and let $P\subset\R^k$ be a rational polytope. Let $m\in\N$ be such that $mP$ is a lattice polytope. 
If $n\in\N$ satisfies $n\ge mk$ and is congruent to 
$r\in\{0,1,\dots,m-1\}$ modulo $m$, then we have the identity 
\[
nP\cap \Z^k=\frac{n-mk-r}{m}*(mP\cap \Z^k)+((mk+r)P\cap\Z^k).
\]
\end{lemma}
\begin{proof}
The proof is an extension of that for Lemma~\ref{geomlemma} and we proceed more briefly.  
Again, it suffices to prove the set inclusion ``$\subseteq$'', the opposite one is trivial.                      
Fix a point $p\in nP\cap\Z^k$ with $n\ge mk$ congruent to $r$ modulo $m$. As in the proof of 
Lemma~\ref{geomlemma}, only replacing the integral part $n_i=\lfloor\beta_i n\rfloor$ with the largest multiple of 
$m$ not exceeding $\beta_i n$, we write $p$ as 
$$
p=\sum_{i=1}^j n_i w_i +\sum_{i=1}^j \alpha_i w_i
$$
where $j\le k+1$, $w_i$ are some vertices of $P$, $n_i\in\No$ are multiples of $m$, $\alpha_i\in[0,m)$, and 
$c=\alpha_1+\dots+\alpha_j=n-(n_1+\dots+n_j)\in\N_0$ is congruent to $r$ modulo $m$.
So $c\le mk+r$. Moving several multiples of $m$ from $n_i$ to the corresponding
$\alpha_i$, we may assume that $c=mk+r$. It follows that the first sum of the right-hand side is equal to an element 
of $\frac{n-mk-r}{m}*(mP\cap\Z^k)$, while the second sum belongs to $(mk+r)P\cap\Z^k$. 
\end{proof}

\noindent
Using this lemma and part 2 of Theorem \ref{summnoz}, we get the following theorem 
in the same way as we got Theorem~\ref{nasepoly}. We omit the proof. 

\begin{theorem}\label{naseqpoly}
Let $P$ be a polytope in $\R^k$ with vertices in $\Q_{\ge 0}^k$, let $m\in\N$ be such that
the vertices of $mP$ lie in $\No^k$, and
let $\chi\colon\No^k\to X$ be an additive coloring. Then, for $n\in\N$ 
sufficiently large, the number of colors
\[
|\chi(nP\cap\No^k)|
\] 
is given by a quasipolynomial with period $m$.
\end{theorem}

\subsection{Strongly eventually polynomial functions}

Theorems \ref{multivar} and \ref{affchov} say nothing about the values of the 
corresponding functions when some argument $n_i$ is not bigger than $c$. In 
Theorem~\ref{stron_ev_poly_chov} we give a stronger formulation using other 
notion of an eventually polynomial function in several variables, which is 
suggested by power series.

For $k,c\in\N$ we define $V(k,c)=([0,c]\cup\{\infty\})^k$; the elements of $V(k,c)$ 
are the $(c+2)^k$ words $w=w_1w_2\dots w_k$ of length $k$ such that every entry
$w_i$ is $0,\dots,c$ or $\infty$. We say that a function 
$$
f\colon\No^k\to\R
$$ 
is {\em strongly eventually polynomial} 
if there exist a $c\in\N$ and $(c+2)^k$ polynomials $p_w\in\R[x_1,\dots,x_k]$ indexed 
by the words $w\in V(k,c)$ so that for every $k$-tuple $n=(n_1,\dots,n_k)\in\No^k$ 
and the unique $w=w(n)\in V(k,c)$ determined by $w_i=n_i$ if $n_i\le c$ and 
$w_i=\infty$ if $n_i>c$, we have
$$
f(n_1,\dots,n_k)=p_{w(n)}(n_1,\dots,n_k).
$$
Said more briefly, there is a constant $c\in\N$ such that for any selection of arguments $n_i$, 
when we fix arguments not exceeding $c$, $f(n_1,\dots,n_k)$ is 
a polynomial function in the remaining arguments (which are all bigger than $c$). 
Note that for $k=1$ this
notion is identical with the usual notion of an eventually polynomial function 
$f\colon\No\to\R$ (there is a constant $c>0$ and a polynomial $p\in\R[x]$ such that $f(n)=p(n)$ for 
$n>c$). Note also that if $f\colon\No^k\to\R$ is strongly eventually polynomial
for a constant $c$, then it is strongly eventually polynomial for any larger constant.

We give a stronger version of Lemma~\ref{event_poly}.

\begin{proposition}\label{event_poly_right}
A function $f\colon\No^k\to\R$ is strongly eventually polynomial if and only if
$$
F(x_1,\dots,x_k)=\sum_{n\in\No^k}f(n)x_1^{n_1}\dotsb x_k^{n_k}
=\frac{r(x_1,\dots,x_k)}{(1-x_1)^{e_1}\dotsb(1-x_k)^{e_k}},
$$
for some $r\in\R[x_1,\dots,x_k]$ and $e_i\in\No$. 
\end{proposition}
\begin{proof}
If $f$ is strongly eventually polynomial and is represented by the polynomials 
$p_v$, $v\in V(k,c)$, we have
$$
F(x)=\sum_{n\in\No^k}f(n)x_1^{n_1}\dotsb x_k^{n_k}=
\sum_{v\in V(k,c)}\ \sum_{n\atop w(n)=v}p_v(n)x_1^{n_1}\dots x_k^{n_k}.
$$
Each inner sum is a power series which can be transformed in the form 
$r(1-x_1)^{-e_1}\dots(1-x_k)^{-e_k}$ for some $r\in\R[x_1,\dots,x_k]$ and $e_i\in\No$. 
Thus $F(x)$ has the stated form.

Suppose that $F(x)$ has the stated form. As in the proof of 
Lemma~\ref{event_poly}, we write it as a linear combination of terms of the 
type $$ \prod_{i=1}^k\frac{x_i^{b_i}}{(1-x_i)^{e_i}}, $$ where 
$b_i,e_i\in\No$. The coefficients of the power series $x^b/(1-x)^e$ form a 
univariate strongly eventually polynomial function. It is easy to see that 
the concatenative product $h\colon\No^{k+l}\to\R$ of two strongly eventually 
polynomial functions $f\colon\No^k\to\R$ and $g\colon\No^l\to\R$, defined by 
$$ h(n_1,\dots,n_{k+l})=f(n_1,\dots,n_k)g(n_{k+1},\dots,n_{k+l}), $$ is 
strongly eventually polynomial as well (as we know, we may assume that the 
constant $c$ is the same for $f$ and $g$). The same holds for the linear 
combination $\alpha f+\beta g\colon\No^k\to\R$ of two strongly eventually 
polynomial functions $f,g\colon\No^k\to\R$. From the expression of $F(x)$ as a 
linear combination of the mentioned products, it follows that the function 
$(n_1,\dots,n_k)\mapsto [x_1^{n_1}\dots x_k^{n_k}]F(x_1,\dots,x_k)$ is a 
finite linear combination of concatenative products of strongly eventually 
polynomial (univariate) functions. Thus it is strongly eventually polynomial 
as well.
\end{proof}

The following theorem is a common strengthening of Theorems~\ref{multivar} 
and \ref{affchov}, which cancels the distinction between the projective and 
affine formulations (parts 1 and 2 of Theorem \ref{summnoz}).     

\begin{theorem}\label{stron_ev_poly_chov}
Let $A_1,\dots,A_l$ be finite sets in a semigroup $(G,+)$. Then
$$
(n_1,\dots,n_l)\mapsto|n_1*A_1+\cdots+n_l*A_l|
$$
is a strongly eventually polynomial function from $\No^l$ to $\No$.
\end{theorem}
\begin{proof}
The proof is almost identical to the proof of Theorem~\ref{multivar}, only we use
Proposition~\ref{event_poly_right} in place of Lemma~\ref{event_poly}.
\end{proof}

\subsection{Multivariate generalizations of Theorems \ref{character} and \ref{iterimage}}

Recall that for $l,c\in\N$, the set $V(l,c)$ consists of the $(c+2)^l$ words of length $l$ 
over the alphabet $\{0,\dots,c,\infty\}$ and that for $n=(n_1,\dots,n_l)\in\No^l$ the word 
$w(n)=w_1\dots w_l\in V(l,c)$ is defined by $w_i=n_i$ if $n_i\le c$ and $w_i=\infty$ if $n_i>c$.
The next theorem generalizes Theorems \ref{character} and \ref{stron_ev_poly_chov} (and thus 
in turn Theorems \ref{summnoz}, \ref{multivar}, and \ref{affchov}).  

\begin{theorem}\label{charactermulti}
For finite sets $A_1,\dots,A_l$ in a semigroup $G=(G,+)$ and a character $\psi\colon G\to\C$, there 
exist a constant $c\in\N$ and $(c+2)^l|A_1|\dotsb|A_l|$ polynomials 
$p_{w,a_1,\dots,a_l}\in\C[x_1,\dots,x_l]$, where $w\in V(l,c)$ and $a_i\in A_i$, such that 
for every $l$-tuple $n=(n_1,\dots,n_l)\in\No^l$ and the corresponding word 
$w(n)\in V(l,c)$, we have  
$$
\sum_{a\in n_1*A_1+\cdots+n_l*A_l}\psi(a)=
\sum_{a_1\in A_1,\dots,a_l\in A_l}p_{w(n),a_1,\dots,a_l}(n_1,\dots,n_l)
\psi(a_1)^{n_1}\dots\psi(a_l)^{n_l}.
$$
\end{theorem}
\begin{proof}[Proof (Sketch)] We pull $\psi$ back to the semigroup 
$(\No^k,+)$ with the associated coloring and for $X\subset\No^k$ work with the 
power series 
$$
F_{X,\psi}(x)=\sum_{n\in X}\psi(n)x_1^{n_1}\dots x_k^{n_k}.
$$
For an orthant $O_s\subset\N_0^k$ we then have, denoting the $k$ basic 
unit vectors by $u_i$,
$$
F_{O_s,\psi}(x)=\frac{\psi(s)x_1^{s_1}\dots 
x_k^{s_k}}{(1-\psi(u_1)x_1)\dots (1-\psi(u_k)x_k)}.
$$
Thus, arguing as in the proof of Theorem \ref{stanleyho}, if $X\subset\No^k$ is a lower 
or an upper ideal, then
$$
F_{X,\psi}(x)=\frac{r(x_1,\dots,x_k)}{(1-\psi(u_1)x_1)\dots (1-\psi(u_k)x_k)}
$$
where $r$ is a polynomial whose coefficients are finite sums of $\pm$ values of $\psi$.
It follows that
$$
\sum_{a\in n_1*A_1+\cdots+n_l*A_l}\psi(a)=[y_1^{n_1}\dots y_l^{n_l}]G
$$
where $G(y)$ is obtained from such $F_{X,\psi}(x)$ by a $P$-substitution. The theorem now follows by a version of Proposition~\ref{event_poly_right} for rational 
power series of the form $r/((1-\alpha_1 x_1)^{e_1}\dots(1-\alpha_k x_k)^{e_k})$.
\end{proof}

\bigskip
In the multivariate generalization of Theorem~\ref{iterimage} we refine 
the iterated image ${\cal F}^{(n)}(B)$ by partitioning 
${\cal F}$. For a (typically infinite) set $X$, its subset $B\subset X$, a family 
${\cal F}$ of mutually commuting mappings $f\colon X\to X$, and a partition $P=\{P_1,\dots,P_l\}$ 
of ${\cal F}$ into nonempty blocks, we let
${\cal F}^{(n_1,\dots,n_l)}$ denote the set of all the functions that can be 
obtained by composing $l$ functions $f_1\circ f_2 \circ\dotsb \circ f_l$, 
where each $f_i$ is itself a composition of $n_i$ functions belonging to the 
block $P_i$, and set 
$$
{\cal F}^{(n_1,\dots,n_l)}(B)=\bigcup_{f\in {\cal F}^{(n_1,\dots,n_l)}}f(B).
$$

The next theorem generalizes Theorems~\ref{iterimage} and \ref{stron_ev_poly_chov} (and thus 
in turn Theorems \ref{summnoz}, \ref{multivar}, and \ref{affchov}).  

\begin{theorem}\label{iterimagemulti}
If $B$ is a finite subset of $X$, ${\cal F}$ is finite family of mutually commuting 
mappings from $X$ to itself, and $P=\{P_1,\dots,P_l\}$ is a partition of ${\cal F}$, then 
$$
(n_1,\dots,n_l)\mapsto|{\cal F}^{(n_1,\dots,n_l)}(B)|
$$ 
is a strongly eventually polynomial function from  $\No^l$ to $\No$.
\end{theorem}

\noindent
For the combinatorial proof we need an extension of Theorem~\ref{stanleyho} to sets more general 
than lower or upper ideals. For $k\in\N$, $I\subset[k]$, and $s\in\No^k$, the 
{\em generalized orthant} $O_{s,I}\subset\No^k$ is defined by
$$
O_{s,I}=\{x\in\No^k:\;i\in I\Rightarrow x_i=s_i,\; i\not\in I\Rightarrow x_i\ge s_i\}.
$$
An empty set is also a generalized orthant. A subset $S\subset\No^k$ is {\em simple} if 
it is a finite union of generalized orthants. In particular, every finite set is simple.
So is every upper ideal and, as we shall see in a moment, every lower ideal. 

\begin{lemma}\label{prasjgenorth}
Intersection of any system of generalized orthants is a generalized orthant. Complement of 
a generalized orthant to $\No^k$ is a simple set.
\end{lemma}
\begin{proof}
A $k$-tuple $x$ of $\No^k$ lies in the intersection of the system $O_{s(j),I(j)}$, $j\in J$, of
nonempty generalized orthants
iff for every $i\in[k]$ the $i$th coordinate $x_i$ satisfies for every $j\in J$ the condition 
imposed by the membership $x\in O_{s(j),I(j)}$. These conditions have form
$x_i\in\{s_{i,j}\}$ or $x_i\in[s_{i,j},+\infty)$ for some $s_{i,j}\in\No$. 
Intersection (conjunction) of these conditions over all $j\in J$ is a condition of the type 
$x_i\in\emptyset$ or $x_i\in\{s_i\}$ or $x_i\in[s_i,+\infty)$ for some $s_i\in\No$. This is 
true for every $i\in[k]$. Thus $\bigcap_{j\in J}O_{s(j),I(j)}$ is an empty set or a nonempty
generalized orthant.

Let $O=O_{s,I}\subset\No^k$ be a generalized orthant. We have $x\in\No^k\backslash O$ iff 
there exists an $i\in[k]$ such that (i) $i\in I$ and $x_i$ satisfies $x_i\in[s_i+1,+\infty)$ or 
$x_i\in[0,s_i-1]$ or such that (ii) $i\not\in I$ and $x_i$ satisfies $x_i\in[0,s_i-1]$. 
Let $u(i,j)\in\No^k$, for $i\in[k]$ and $j\in\No$, denote the $k$-tuple with all coordinates zero except the $i$th one which is equal to $j$. It follows that  
$\No^k\backslash O$ is the union of the generalized orthants 
$$
O_{u(i,s_i+1),\emptyset},\ i\in I;\ O_{u(i,j_i),\{i\}},\ i\in [k]\ \mbox{ and }\ j_i\in[0,s_i-1]
$$ 
(if $s_i=0$, no $O_{u(i,j_i),\{i\}}$ is needed). Thus $\No^k\backslash O$ is a simple set.
\end{proof}

\begin{corollary}\label{booleanalg}
The family of simple sets in $\No^k$ contains the sets $\emptyset$ and $\No^k$ and is 
closed under taking finite unions, finite intersections, and complements. Hence it forms 
a boolean algebra. 
\end{corollary}
\begin{proof}
This follows by the previous lemma and by elementary set identities involving unions, intersections and complements. 
\end{proof}

\noindent
The family of simple sets is in general not closed to infinite unions nor to infinite intersections. 

The next theorem is an extension of Theorems \ref{dolniid} and \ref{stanleyho}. It 
characterizes the sets $S\subset\No^K$, for which these theorems hold.

\begin{theorem}\label{zobstanleyho}
If $S\subset\No^k$ is a simple set, then
$$
F_S(x_1,\dots,x_k)=\frac{r(x_1,\dots,x_k)}{(1-x_1)\dotsb (1-x_k)}
$$
where $r(x_1,\dots,x_k)$ is an integral polynomial. If $S\subset\No^k$
is a set such that 
$$
F_S(x_1,\dots,x_k)=\frac{r(x_1,\dots,x_k)}{(1-x_1)\dotsb (1-x_k)}
$$
where $r(x_1,\dots,x_k)$ is an integral polynomial, then $S$ is a simple set.
\end{theorem}
\begin{proof}
Suppose that $S\subset\No^k$ is simple and $S=O_1\cup\dotsb\cup O_r$ for some generalized 
orthants $O_i$. By the principle of inclusion and exclusion, $F_S(x_1,\dots,x_k)$ is a sum
of the $2^r$ terms $(-1)^{|X|}F_{O(X)}(x_1,\dots,x_k)$, $X\subset[r]$, where
$$
O(X)=\bigcap_{i\in X}O_i.
$$
By Lemma~\ref{prasjgenorth}, each $O(X)$ is again a generalized orthant. For a generalized orthant $O=O_{s,I}$, 
$$
F_O(x_1,\dots,x_k)=\frac{x_1^{s_1}\dotsb x_k^{s_k}}{\prod_{i\in [k]\backslash I}(1-x_i)}.
$$
The first claim follows.

Suppose that $S\subset\No^k$ and $F_S(x_1,\dots,x_k)=r/((1-x_1)\dotsb(1-x_k))$ where 
$r\in\Z[x_1,\dots,x_k]$. Hence $F_S(x_1,\dots,x_k)$ is an $l$-term integral linear combination 
$$
\sum_{s\in T}\frac{c_sx_1^{s_1}\dotsb x_k^{s_k}}{(1-x_1)\dotsb(1-x_k)}
$$
where $T\subset\No^k$, $|T|=l$, and $c_s\in\Z$. Every summand is in fact equal to 
$c_sF_{O_s}(x_1,\dots,x_k)$. The 
characteristic function of $S$ is an integral linear combination of the characteristic functions 
of the $l$ (full-dimensional) orthants $O_s=O_{s,\emptyset}$, $s\in T$. With $X$ running 
through the $2^l$ subsets of $T$, we partition $\No^k$ in the $2^l$ cells
$$
\bigcap_{s\in X}O_s\cap\bigcap_{s\in T\backslash X}\No^k\backslash O_s.
$$
The characteristic function of $S$ is an integral linear combination of the characteristic functions of these cells. Since the cells are pairwise disjoint, it follows that $S$ is 
a union of some of these cells. Each cell is a simple 
set by Corollary~\ref{booleanalg} and therefore $S$ is a simple set as well. 
\end{proof}

\begin{proof}[Proof of Theorem~\ref{iterimagemulti}] 
Let $X$, $B$, ${\cal F}$, and $P=\{P_1,\dots,P_l\}$ be as stated. Enlarging 
${\cal F}$ by repeating some mappings and enlarging $B$ by repeating some 
elements does not affect the set ${\cal F}^{(n_1,\dots,n_l)}(B)$. Therefore, 
we may assume that $|{\cal F}|=|B|=k$, ${\cal F}=\{f_1,\dots,f_k\}$ and 
$B=\{b_1,\dots,b_k\}$. We set $K=k^2$ and define a partial coloring
$$
\chi\colon\No^K=\No^{k^2}\to X\cup\{u\}
$$
as follows: the elements $x$ with $\chi(x)=u$ are regarded as ``uncolored''; for $i\in[k]$ and
$x\in\No^K$ such that $z_1:=x_{(i-1)k+1},\dots,z_k:=x_{(i-1)k+k}$ are positive but all other coordinates of $x$ are zero, we set
$$
\chi(x)=(f_1^{z_1-1}\circ\dots\circ f_k^{z_k-1})(b_i).
$$
Note that if $z_1=\dots=z_k=1$, then $\chi(x)=b_i$. We denote the set of all these points 
$x$ by $C_i$. The set of colored points is $C=C_1\cup\dots\cup C_k$. 
The points in $\No^K\backslash C$ are uncolored.
Each $C_i$ is a generalized orthant. If
$x\in C_i$ and $x'\in C_j$ for $i<j$, then $x$ and $x'$ are incomparable by $\le$ but 
$x'<_{lex}x$. For $x\in\No^K$ with all coordinates different from 
$(i-1)k+1,\dots,(i-1)k+k$ equal to zero (e.g., if $x\in C_i$) and $j\in[k]$, we define $x(j)$ by shifting the 
$k$-term block of possibly nonzero coordinate values to the coordinates 
$(j-1)k+1,\dots,(j-1)k+k$. The key property of $\chi$ is the following:
$$
\mbox{if $x,y\in C_i$, $x\le y$, $x'\in C_j$, and $\chi(x)=\chi(x')$, then 
$\chi(y)=\chi(x'+(y-x)(j))$.}
$$  
Indeed, if $\chi(x)=\chi(x')=c\in X$ and the coordinates $k(i-1)+1,\dots,k(i-1)+k$ of $y-x$ 
are $z_1,\dots,z_k$, then 
$\chi(y)=\chi(x+(y-x))=(f_1^{z_1}\circ\dots\circ f_k^{z_k})(c)=\chi(x'+(x-y)(j))$.

$P$ induces naturally a partition of $[K]$ into $l$ blocks which we again denote $P=\{P_1,\dots,P_l\}$: for $f_j\in P_r$ we put in the $P_r\subset[K]$ all $k$ elements 
$j,j+k,j+2k,\dots,j+(k-1)k$. Note that for $n_1,\dots,n_l\in\N$ we have (recall the definition 
of $\|x\|_P$ before the proof of Corollary~\ref{NRgf})
$$
|\chi(\{x\in\No^K\colon\|x\|_P=(n_1,\dots,n_l)\})\backslash\{u\}|=
|{\cal F}^{(n_1-1,\dots,n_l-1)}(B)|.
$$
We call a point $x\in\No^K$ $P$-\emph{substantial} if it is colored and is the lexicographically minimum element in the set 
$$
\{y\in\No^K\colon\chi(y)=\chi(x), \|y\|_P=\|x\|_P\}. 
$$
As before, $P$-substantial points are representatives of the nonempty intersections of the color 
classes of $\chi$ with the simplex $\|x\|_P=(n_1,\dots,n_l)$. Thus
$$
|{\cal F}^{(n_1-1,\dots,n_l-1)}(B)|=[y_1^{n_1}\dots y_1^{n_l}]G
$$
where $G(y_1,\dots,y_l)$ is obtained by the $P$-substitution from $F_S(x_1,\dots,x_K)$ and 
$S$ is the set of all $P$-substantial points in $\No^K$. Now the theorem 
follows as before by Proposition~\ref{event_poly_right}, Lemma~\ref{bsubst} and 
Theorem~\ref{zobstanleyho}, provided that we show that $S$ is a simple set.

To prove that $S$ is simple we consider the complement $\No^K\backslash S$. We have that
$$
\No^K\backslash S=(\No^K\backslash C)\cup C^*
$$
where $C^*$ consists of all colored points that are not $P$-substantial. The set 
$\No^K\backslash C$ is simple by Corollary~\ref{booleanalg} because $C$ is simple (as a union 
of the generalized orthants $C_i$). Now $C^*=C_1^*\cup\dots\cup C_k^*$ where $C_i^*=C^*\cap C_i$.
We show that each $C_i^*$ is an upper ideal in $(C_i,\le)$. Then, by Dickson's lemma, $C_i^*$ 
is a finite union of generalized orthants, which implies that $C_i^*$ and $C^*$ 
are simple. So $\No^K\backslash S$ is simple and $S$ is simple.

Thus suppose that $x\in C^*_i$ and $y\in C_i$ with $x\le y$. It follows that there is a colored 
point $x'\in\No^K$ with $\chi(x')=\chi(x)$,  $\|x'\|_P=\|x\|_P$, and $x'<_{lex}x$. Let 
$x'\in C_j$. Consider the point $y'=x'+(y-x)(j)$. By the property of $\chi$ we have 
$\chi(y')=\chi(y)$. Since $\|y-x\|_P=\|(y-x)(j)\|_P$ (by the definition of $P$), we have
$\|y'\|_P=\|x'+(y-x)(j)\|_P=\|x\|_P+\|(y-x)(j)\|_P=\|x\|_P+\|y-x\|_P=\|y\|_P$. If $i=j$, then 
$y-x=(y-x)(j)$ and $y'=x'+(y-x)<_{lex}x+(y-x)=y$. If $i\ne j$, we must have $i<j$ because 
$x'\in C_j$, $x\in C_i$, and $x'<_{lex}x$. But $y'\in C_j$ and $y\in C_i$, so again $y'<_{lex}y$. 
Thus $\chi(y')=\chi(y)$, $\|y'\|_P=\|y\|_P$, and $y'<_{lex}y$, which shows that $y\in C^*_i$. 
We have shown that $C_i^*$ is an upper ideal in $(C_i,\le)$, which concludes the proof. 
\end{proof}

\section{Concluding remarks}

In \cite{jeli_klaz_next}, we plan to look from general perspective at further polynomial and quasipolynomial classes of enumerative problems. A natural question, for 
example, is about the multivariate generalization of Theorem \ref{nasepoly}; generalization 
of Theorem \ref{mrbody} to several variables was considered by Beck \cite{beck99, beck02}. 
Theorem \ref{nasepoly} is related in spirit to results of Lison\v ek \cite{liso} who counts 
orbits of group actions on lattice points in polytopes. It would be interesting to have 
an explicit description of the structure of an additive coloring 
$\chi\colon\No^k\to X$ because one may consider further statistics of $\chi$ on the points 
$nP\cap\Z^k$, such as the number of occurrences of a specified color. We plan to investigate
polynomial classes arising from counting permutations (e.g., Albert, 
Atkinson and Brignall \cite{albe_atki_brig}, Huczynska and Vatter \cite{hucz_vatt}, 
Kaiser and Klazar \cite{kais_klaz}), graphs (e.g., Balogh, 
Bollob\'as and Morris \cite{balo_boll_morr}), relational structures (e.g., Pouzet and Thi\'ery 
\cite{pouz_thie}), and perhaps other.
\end{document}